\documentclass{article}
\usepackage{graphicx} 
\usepackage{amsthm} 
\usepackage{amssymb}
\newtheorem{theorem}{Theorem}
\newtheorem{corollary}{Corollary}
\newtheorem{lemma}{Lemma}

\newtheorem*{conjecture}{Conjecture}
\usepackage{amsmath}
\usepackage{float}
\usepackage{subfigure}
\usepackage{tikz}
\usepackage{titlesec}
\titlelabel{\thetitle.\quad}
\usepackage{comment}

\title{Isomorphism of Clean Graphs over $\mathbb{Z}_n$ and Structural Insight into $M_2(\mathbb{Z}_p)$}
\author{Felicia Servina Djuang$^{1}$, Indah Emilia Wijayanti$^{2}$, and Yeni Susanti$^{3}$\\
{\small $^{1,2,3}$Department of Mathematics, Universitas Gadjah Mada, Yogyakarta, Indonesia}\\
\small{$^{1}$feliciadjuang25@mail.ugm.ac.id, $^{2}$ind\_wijayanti@ugm.ac.id, $^{3}$yeni\_math@ugm.ac.id}}
\date{}

\begin{document}
\maketitle
\begin{abstract}
Let $R$ be a finite ring with identity. The clean graph $Cl(R)$ of a ring $R$ is a graph whose vertices are pairs $(e, u)$, where $e$ is an idempotent element and $u$ is a unit of $R$. Two distinct vertices $(e, u)$ and $(f, v)$ are adjacent if and only if $ef = fe = 0$ or $uv = vu = 1$. The graph $Cl_2(R)$ is the induced subgraph of $Cl(R)$ induced by the set $\{(e, u): e \text{ is a nonzero idempotent and } u \text{ is a unit of } R\}$. In this study, we present properties that arise from the isomorphism of two clean graphs and conditions under which two clean graphs over direct product rings are isomorphic. We also examine the structure of the clean graph over the ring $M_2(\mathbb{Z}_p)$ through their $Cl_2$ graph. \\
\textbf{Keyword:} clean graph, idempotent graph, isomorphism graph, unit, idempotent.\\
\textbf{2020 AMS Subject Classification:} 05C60, 05C25, 13A70, 16U60, 16U40.
\end{abstract}

\section{Introduction}
The interaction between algebra and graph theory has led to profound deve\-lopments in modern mathematics. This connection was first explored by Beck in \cite{beck}, where the elements of a commutative ring $R$ were represented as vertices of a graph, with a focus on coloring based on zero-divisors. Building on this idea, Anderson and Livingston \cite{anderson} introduced the zero-divisor graph $\Gamma(R)$, whose vertices are the nonzero zero divisors of $R$, with two distinct vertices $x$ and $y$ adjacent if and only if $xy = yx = 0$. Since then, a variety of algebraic graphs have been proposed to capture the intricate relationship between ring-theoretic properties and combinatorial structures \cite{wilson}. Notable examples include the unit graph over $\mathbb{Z}_n$, introduced by Grimaldi \cite{grimaldi}, where two vertices $x$ and $y$ are adjacent if $x + y$ is a unit in $\mathbb{Z}_n$, and the idempotent graph introduced by Akbari et al. \cite{akbari}, whose vertices are the nontrivial idempotent elements of a ring $R$, with adjacency defined by $xy = yx = 0$.

These constructions not only deepen our understanding of algebraic objects but also serve as bridges to applications in diverse areas such as coding theory \cite{fish,jain}, cryptography, and network theory. Within this expanding landscape, the search for new algebraic graphs that reveal hidden structural insights has become both natural and necessary.

One such construction is the clean graph of a ring, introduced by Habibi et al. \cite{habibiyet}. This graph is motivated by the notion of clean rings \cite{nicholzhou,immormino}, where each element decomposes into an idempotent and a unit. The clean graph over the ring $R$, denoted by $Cl(R)$, has as vertices pairs of an idempotent and a unit from the ring $R$, and two vertices $(e,u)$ and $(f,v)$ are adjacent if either $ef=fe=0$ or $uv=vu=1$. Furthermore, $Cl_1(R)$ and $Cl_2(R)$ are induced subgraphs of $Cl(R)$ that are induced by $\{(0,u): u \text{ is unit in } R\}$ and $\{(e,u): 0 \neq e \text{ is idempotent and } u \text{ is unit in } R\}$ respectively.
We begin by reviewing some fundamental concepts that will be used throughout this paper. Let $R$ be a ring with identity. An element $e \in R$ is called an idempotent if $e^2 = e$, while an element $u \in R$ is called a unit if there exists $v \in R$ such that $uv = vu = 1$. We denote the set of all idempotents in $R$ by $Id(R)$, and the set of all units by $U(R)$. Moreover, the set $U(R)$ can be divided into two disjoint subsets: $U'(R)=\{u \in U(R): u^2=1\} \text{ and } U''(R) =U(R) \setminus U'(R).$
For any other ring-theoretic notation or background on clean rings, we refer the reader to \cite{malikmor,nicholzhou}.

Let $G = (V(G), E(G))$ be a graph, where $V(G)$ and $E(G)$ denote its vertex and edge sets, respectively. The degree of a vertex $v \in V(G)$, written as $\deg_G(v)$, is the number of edges in $G$ incident with $v$. Two graphs $G_1$ and $G_2$ are said to be isomorphic, written $G_1 \cong G_2$, if there exists a bijection $f: V(G_1) \to V(G_2)$ such that for every pair of vertices $u, v \in V(G_1)$, the vertices $u$ and $v$ are adjacent in $G_1$ if and only if $f(u)$ and $f(v)$ are adjacent in $G_2$.

Another graph operation relevant to our discussion is the shuriken graph, introduced in \cite{djuang}. Let $n,t$ be positive integers such that $n-t$ is even. The $(t,n)$-shuriken graph of $G$, denoted $Shu^t_n(G)$, is constructed as follows: first add a new vertex $z$ to $G$, then take $n$ copies of the resulting graph. Denote by $G'_i$ the $i$-th copy of this graph for $1 \leq i \leq n$. The vertex and edge sets of $Shu^t_n(G)$ are then defined by
$$V(Shu^t_n(G))=\bigcup_{i=1}^n \{z_i, v_i: v \in V(G)\} \text{ and}$$
    \begin{align*}
        E(Shu^t_n(G))= &\{u_iv_j : uv \in E(G), i,j \in \{1,2,\dots,n\}\} \\
        &\cup \{u_iv_i: u_i,v_i \in V(G'_i), u_i\neq v_i, i \in\{1,2,\dots,t\}\}\\ 
        &\cup \Bigg\{u_iv_{n+t+1-i}: u_i \in V(G'_i), v_{n+t+1-i} \in V(G'_{n+t+1-i}) \\
        & \hspace{2.5 cm}i \in \left\{t+1,t+2,\dots,\frac{n+t}{2}\right\} \Bigg\}.\end{align*}
For general background on graph theory and terminology, see \cite{wilson}.

In \cite{djuang}, it was shown that for any ring $R$ with identity, $Cl_2(R) \cong Shu^t_k(I(R))$ where $t = |U'(R)|$ and $k = |U(R)|$. As a consequence, the structure of the clean graph $Cl_2$ over $\mathbb{Z}_n$ was determined. 

A particularly pressing problem is the classification and isomorphism of clean graphs. Isomorphism results in algebraic graph theory are far from simple technicalities; they provide deep insight into how ring-theoretic similarities manifest in graph structures. For example, two nonisomorphic rings may give rise to isomorphic algebraic graphs, or conversely, graph-theoretic distinctions may reveal subtle algebraic differences. In practical contexts, graph isomorphism is related to recognition problems in discrete mathematics, efficient coding constructions, and structural equivalences in networks \cite{fish,jain}. Despite their centrality in the study of other algebraic graphs, such results for clean graphs remain scarce. Further research on clean graph isomorphisms was conducted by \cite{pongthana}, who showed that for any prime numbers $p, q$ and any integer $r$ with $\gcd(p,r) = \gcd(q,r) = 1$, $Cl_2(\mathbb{Z}_{p^nr}) \cong Cl_2(\mathbb{Z}_{q^mr}) \iff p^n-p^{n-1} = q^m-q^{m-1}$. Without a systematic understanding of their isomorphism classes, the development of clean graph theory will remain incomplete.

Another dimension of urgency arises from matrix rings, particularly $M_2(\mathbb{Z}_p)$. These rings exhibit highly nontrivial unit groups and nontrivial idempotent elements, making them a natural laboratory for exploring the full complexity of clean graphs. Although zero-divisor and idempotent graphs of matrix rings have been examined \cite{patil}, the clean graph of matrix rings has barely been touched upon. The subgraph $Cl_2(R)$, induced by nonzero idempotents, provides a more refined tool to probe these structures. Studying $Cl_2(M_2(\mathbb{Z}_p))$ not only addresses a gap in the literature but also reveals interactions between matrix idempotents and units that previously studied graphs cannot capture.

By establishing conditions under which clean graphs over direct product rings are isomorphic, and by analyzing the structure of the clean graph over $M_2(\mathbb{Z}_p)$ through its $Cl_2$ subgraph, we contribute to the foundational development of clean graph theory. Our work not only extends existing results but also introduces new perspectives on the role of idempotents and units in algebraic graph theory. In doing so, it opens the door for broader applications of clean graphs, both in the classification of algebraic structures and in potential applied domains such as cryptography, network design, and error-correcting codes.
    
\section{Isomorphism of Clean Graphs}

For any isomorphic rings, the relationship between the clean graphs over these rings is considered and presented in the following lemma.
\begin{lemma}
    Given rings $R_1$ and $R_2$ with identity elements and satisfying $R_1 \cong R_2$, it holds that
    $$Cl(R_1) \cong Cl(R_2) \text{ and } Cl(nR_1) \cong Cl(nR_2), \text{ for all } n \in \mathbb{N},$$
    where $nR=\underbrace{R \times R \times \dots \times R}_{n \text{ times}}$.
\end{lemma}
\begin{proof}
Since $R_1 \cong R_2$, there exist bijective functions $f_1: Id(R_1) \to Id(R_2)$ and $f_2: U(R_1) \to U(R_2)$. Since $$V(Cl(R_1))=Id(R_1) \times U(R_1) \text{ and } V(Cl(R_2))=Id(R_2) \times U(R_2),$$ a graph isomorphism can be defined as $$g: V(Cl(R_1)) \to V(Cl(R_2)) \text{ with } g((e,u))=(f_1(e),f_2(u))$$ for each $(e,u) \in V(Cl(R_1))$. \\
Let $(e_1,u_1),(e_2,u_2) \in V(Cl(R_1))$ be arbitrary. Observe that
\begin{align*}
    (e_1,u_1)(e_2,u_2) \in E(Cl(R_1)) &\iff e_1e_2=0 \text{ atau } u_1u_2=1\\
    &\iff f_1(e_1e_2)=0 \text{ atau } f_2(u_1,u_2)=1\\
    &\iff f_1(e_1)f_1(e_2)=0 \text{ atau } f_2(u_1)f_2(u_2)=1\\
    &\iff (f_1(e_1),f_2(u_1))(f_1(e_2),f_2(u_2)) \in E(Cl(R_2))\\
    &\iff g((e_1,u_1))g((e_2,u_2)) \in E(Cl(R_2)).
\end{align*}   
For any $n\in \mathbb{N}$, we know that $nR_1 \cong nR_2$. Thus,
$$Cl(R_1) \cong Cl(R_2) \text{ and } Cl(nR_1) \cong Cl(nR_2), \text{ for all } n \in \mathbb{N}.$$
\end{proof}

Consider the clean graph for the rings $\mathbb{Z}_3$ and $\mathbb{Z}_4$, as well as the clean graph for the rings $\mathbb{Z}_7$ and $\mathbb{Z}_9$. We have
\begin{align*}
    Cl_2(\mathbb{Z}_3) \cong Cl_2(\mathbb{Z}_4) = 2 K_1, \text{ and }\\
    Cl_2(\mathbb{Z}_7) \cong Cl_2(\mathbb{Z}_9) = 2 K_1 \cup 2 K_2.
\end{align*}
Based on the examples above, there exist rings $R$ and $S$ such that $Cl_2(R) \cong Cl_2(S)$, but $R \ncong S$. For the rings $\mathbb{Z}_n$ and $\mathbb{Z}_m$ with natural numbers $n,m$, it is known that $\mathbb{Z}_n \cong \mathbb{Z}_m$ if and only if $n=m$, which implies that $|\mathbb{Z}_n|=|\mathbb{Z}_m|$. Thus, trivially, for any modular integer rings $R_1, R_2$, $Cl_2(R_1) \cong Cl_2(R_2)$ and $|R_1|=|R_2|$ if and only if $R_1 \cong R_2$. 

However, this is not true for every ring $R_1$ and ring $R_2$, because there exist rings $\mathbb{Z}_4$ and $\mathbb{Z}_2[x]/\langle x^2 \rangle$ such that $|\mathbb{Z}_4|=4=|\mathbb{Z}_2[x]/\langle x^2 \rangle|$ and $Cl_2(\mathbb{Z}_4) \cong Cl_2(\mathbb{Z}_2[x]/\langle x^2 \rangle)$, but $\mathbb{Z}_4 \ncong \mathbb{Z}_2[x]/\langle x^2 \rangle$.

Furthermore, consider that
$$Cl_2\left(\mathbb{Z}_3 \times \mathbb{Z}_3\right) \cong Cl_2\left(\mathbb{Z}_3 \times \mathbb{Z}_4\right) \cong Cl_2\left(\mathbb{Z}_4 \times \mathbb{Z}_4\right) \cong Cl_2(\mathbb{Z}_{12}).$$ In this case, it will be proven that for any rings $\mathbb{Z}_{p^n}$ and $\mathbb{Z}_{q^m}$ with $Cl_2(\mathbb{Z}_{p^n}) \cong Cl_2(\mathbb{Z}_{q^m})$, it holds that $Cl_2(\mathbb{Z}_{p^n} \times \mathbb{Z}_k) \cong Cl_2(\mathbb{Z}_{q^m} \times \mathbb{Z}_k)$ for some prime numbers $p,q$ and natural numbers $n,m,k$. Before that, we first investigate the relationship between the isomorphism of the clean graphs $Cl(R)$ and $Cl(S)$ and the isomorphism of their corresponding graphs $Cl_2(R)$ and $Cl_2(S)$. We show that these two notions of isomorphism are equivalent, and we further discuss additional properties that arise from this equivalence.

\begin{theorem}\label{isomorClCl2}
    Given finite rings with identity element $R$ and $S$, the graph $Cl(R) \cong Cl(S)$ if and only if $Cl_2(R) \cong Cl_2(S)$.
\end{theorem}
\begin{proof}
    Suppose $Cl_2(R) \cong Cl_2(S)$. Since $Cl(R)=Cl_1(R)+Cl_2(R)$ and $Cl(S)=Cl_1(S)+Cl_2(S)$, assume for contradiction that $Cl(R) \ncong Cl(S)$, then it must be that $Cl_1(R) \ncong Cl_1(S)$. Since $Cl_1(R)$ and $Cl_1(S)$ are complete graphs, we obtain $K_{|U(R)|} \neq K_{|U(S)|}$. Consequently, $|U(R)| \neq |U(S)|$. Since $|V(Cl_2(R))|=|V(Cl_2(S))|$, we have
    \begin{align*}
        &|Id(R) \setminus \{0\} \times U(R)| = |Id(S) \setminus \{0\} \times U(S)|\\ \iff &|Id(R) \setminus \{0\}| |U(R)| = |Id(S) \setminus \{0\}| |U(S)|. 
    \end{align*} 
    Thus, it must be that $|Id(R) \setminus \{0\}| \neq |Id(S) \setminus \{0\}|$. Suppose 
    \begin{align*}
        |Id(R) \setminus \{0\}|=n, \quad 
        |Id(S) \setminus \{0\}|=m, \quad |U(R)|=k, \quad |U(S)|=l
    \end{align*}
    where $(n,m),(k,l)$ are distinct pairs of natural numbers. We consider the following two cases:
    \begin{enumerate}
        \item Case $n < m$: Similarly, the vertex $(1,1)$ in $Cl_2(R)$ has degree $n-1$. On the other hand, for every vertex $(e,u) \in V(Cl_2(S))$, we have
        $$deg_{Cl_2(S)}(e,u)=\begin{cases}
        m-1 + O_e \left( l-1 \right), &\text{if } u \in U'(R),\\ 
        m + O_e \left( l-1 \right), &\text{if } u \in U''(R). 
    \end{cases}$$ Since $O_e \geq 0$, we obtain $deg_{Cl_2(S)}(e,u) > deg_{Cl_2(R)}(1,1)$. This contradicts the fact that $Cl_2(R) \cong Cl_2(S)$.

        \item Case $n > m$: Similarly, the vertex $(1,1)$ in $Cl_2(S)$ has degree $m-1$. On the other hand, for every vertex $(e,u) \in V(Cl_2(R))$, we have
        $$deg_{Cl_2(R)}(e,u)=\begin{cases}
        n-1 + O_e \left( k-1 \right), &\text{if } u \in U'(R),\\ 
        n + O_e \left( k-1 \right), &\text{if } u \in U''(R). 
    \end{cases}$$ Since $O_e \geq 0$, we obtain $deg_{Cl_2(R)}(e,u) > deg_{Cl_2(S)}(1,1)$. This contradicts the fact that $Cl_2(R) \cong Cl_2(S)$.
    \end{enumerate}
    Since both cases lead to contradictions, it follows that $Cl(R) \cong Cl(S)$. On the other hand, if $Cl(R) \cong Cl(S)$ and $Cl_2(R) \ncong Cl_2(S)$, it must be that $Cl_1(R) \ncong Cl_1(S)$. Consequently, $|Id(R)| \neq |Id(S)|$ and $|U(R)|\neq |U(S)|$. Without loss of generality, assume $|U(R)|>|U(S)|$. Suppose $|U(S)|=1$, we get $Cl(S) = K_{|Id(S)|}$. Although, there is $1 \neq v \in U(R)$ such that $(1,1)$ and $(1,v)$ are not adjacent. It is a contradiction. Thus, $|U(R)|>|U(S)|>1$. It means that there exists a vertex $x=(e,u)$ in $Cl_2(S)$ such that $x$ is adjacent to all vertices in $Cl(S)$. Thus $(e,u)$ and $(1,1)$ are adjacent. Consequently, $e=0$ or $u=1$. Suppose $u=1$, there is $1 \neq w \in U(S)$ such that $(1,u)$ and $(e,u)$ are not adjacent. Hence, it should be $e=0$. Inconsistent with the fact that $x \in Cl_2(S)$. Therefore, it must be $|U(R)| =  |U(S)|$. Furthermore, $|Id(R)|=|Id(S)|$. Clearly that $Cl_2(R) \cong Cl_2(S)$.
\end{proof}

\begin{corollary}\label{equalisoIdU}
    Given finite rings with identity element $R$ and $S$ with $Cl_2(R) \cong Cl_2(S)$, the following statements hold:
    \begin{enumerate}
        \item $|Id(R)|=|Id(S)|$
        \item $|U(R)|=|U(S)|$.
    \end{enumerate}
\end{corollary}

\begin{lemma}\label{equalisoU'}
    Given finite rings with identity element $R$ and $S$ with $|U(R)| > 1$, such that $Cl_2(R) \cong Cl_2(S)$. Let $f: V(Cl_2(R)) \to V(Cl_2(S))$ be the graph isomorphism function. For any $(a,c) \in V(Cl_2(R))$, let $f(a,c)=(b,d)$, the following hold:
    \begin{align*}
        O_a=O_b \text{ and } (c,d) \in (U'(R) \times U'(S)) \cup (U''(R) \times U''(S)).
    \end{align*}
    Furthermore, $|U'(R)|=|U'(S)|$.
\end{lemma}
\begin{proof}
   Let $$|Id(R)|=|Id(S)|=n+1 \text{ and } |U(R)|=|U(S)|=k>1.$$ For any $(a,c) \in V(Cl_2(R))$, let $f(a,c)=(b,d)$. Using the degree formula in \cite{djuang}, there are four possible cases:
    \begin{enumerate}
        \item Case $c \in U'(R)$ and $d \in U''(S)$. We get $O_a(k-1)=O_b(k-1)+1 \iff O_a=O_b+ \frac{1}{k-1}$. Since $O_a,O_b \in \mathbb{Z}^+$, it must be $k=2$. However, since $1 \in U'(S)$ and $d \in U''(S)$, it follows that $d \neq 1$ and there exists $d' \notin {1,d}$ such that $dd' = 1$. Hence, $k = 2$ is impossible.
        \item Case $c \in U''(R)$ and $d \in U'(S)$. We get $O_a(k-1)=O_b(k-1)-1 \iff O_a=O_b- \frac{1}{k-1}$. Since $O_a,O_b \in \mathbb{Z}^+$, it must be $k=2$. However, since $1 \in U'(R)$ and $c \in U''(R)$, it follows that $c \neq 1$ and there exists $c' \notin {1,c}$ such that $cc' = 1$. Thus, it is impossible for $k = 2$.
        \item Case $(c,d) \in (U'(R) \times U'(S)) \cup (U''(R) \times U''(S))$. We get $O_a(k-1)=O_b(k-1)$. Since $k>1$, it follows that $O_a=O_b$.
    \end{enumerate}
    Furthermore, for any $(1,x) \in V(Cl_2(R))$, let $f(1,x)=(y,z)$. We know that $0=O_1=O_y$, which implies that $y$ is not a zero divisor. Since $R$ is a finite ring, we conclude that $y$ is a unit. It follows that $y^2=y$ and there exists $y' \in S$ such that $yy'=1$. Thus,
    $$y=y(yy')=(yy)y'=y^2y'=yy'=1.$$ 
    On the other hand, $(x,z) \in (U'(R) \times U'(S)) \cup (U''(R) \times U''(S))$. Since $f$ is a bijective function and $|U(R)|=|U(S)|$, we have $|U'(R)|=|U'(S)|$.
\end{proof}

We now examine the conditions under which two clean graphs over $\mathbb{Z}_n$ are isomorphic, focusing on the case where $n$ is a natural number with a single prime factor.
\begin{lemma}\label{lemaisoZpnZqm}
    Let $p,q$ be distinct prime numbers and $n,m \in \mathbb{N}$. Then 
    \begin{align*}
        Cl_2(\mathbb{Z}_{p^n}) \cong Cl_2(\mathbb{Z}_{q^m}) \iff &(\{p^n,q^m\}=\{2^2,3^1\}) \vee \\
        &(p,q \neq 2 \wedge p^n-p^{n-1} = q^m - q^{m-1}).
    \end{align*}
\end{lemma}
\begin{proof}
    ($\Rightarrow$) Suppose $Cl_2(\mathbb{Z}_{p^n}) \cong Cl_2(\mathbb{Z}_{q^m})$. From their structure, we consider the following cases:
    \begin{enumerate}
        \item Exactly one of $p,q$ is equal to $2$; without loss of generality, assume $p=2$. It must be that $n=2$, hence $\frac{q^m-q^{m-1}}{2}-1=0 \iff q^{m-1}(q-1)=2$. This results in two possibilities. First, $q^{m-1}=2$ and $q-1=1$, which gives contradicts, since $q=2$. Second, $q^{m-1}=1$ and $q-1=2$. So, we get $q=3$ dan $m=1$.
        \item If $p,q \neq 2$, then we must have $\frac{p^n-p^{n-1}}{2}-1=\frac{q^m-q^{m-1}}{2}-1$. This simplifies to $p^n-p^{n-1}=q^m-q^{m-1}$.
    \end{enumerate}
    ($\Leftarrow$) Given that $\{p^n,q^m\}=\{2^2,3^1\}$ or $p,q \neq 2$ with $p^n-p^{n-1} = q^m - q^{m-1})$. We refer back to Theorem 2 in \cite{djuang} and consider the following cases:
    \begin{enumerate}
        \item If $\{p^n,q^m\}=\{2^2,3^1\}$, then without loss of generality, assume $p^n=2^2=4$ and $q^m=3^1=3$. We obtain
        $Cl_2(\mathbb{Z}_{p^n}) = 2K_1 = Cl_2(\mathbb{Z}_{q^m})$.
        \item If $p,q \neq 2$ and $p^n-p^{n-1} = q^m - q^{m-1}$, then
        \begin{align*}
            Cl_2(\mathbb{Z}_{p^n}) &= 2K_1 \cup \left(\frac{p^n-p^{n-1}}{2}-1 \right) K_2 \\&= 2K_1 \cup \left(\frac{q^m-q^{m-1}}{2}-1 \right) K_2 = Cl_2(\mathbb{Z}_{q^m}).
        \end{align*}
    \end{enumerate}
\end{proof}

Based on the lemmas above, the following theorems are presented.
\begin{theorem}\label{teocorZpnpm}
    For any odd prime number $p$ and a natural number $n>1$, the following holds
	\begin{align*}
		Cl_2(\mathbb{Z}_{p^n}) \cong Cl_2(\mathbb{Z}_{q^m}) \iff &q=p^n-p^{n-1}+1 \text{ is prime number} \text{ and } m=1.
	\end{align*}
\end{theorem}
\begin{proof}
    For any odd prime number $p$ and a natural number $n>1$, if $Cl_2(\mathbb{Z}_{p^n}) \cong Cl_2(\mathbb{Z}_{q^m})$, based on the Theorem \ref{lemaisoZpnZqm}, $q \neq 2$ and $p^n-p^{n-1}=q^m-q^{m-1}$. Assume that $m \neq 1$, it means $m>1$. Let $p=2a+1$ and $q=2b+1$ for some $a,b \in \mathbb{Z}^+$. Consequently, $p^{n-1}a=q^{m-1}b$. Moreover, $p>a$ and $q>b$. Hence, $\gcd(p,a)=\gcd(q,b)=1$.
	If $p>q$, then $a>b$ and $\gcd(p,b)=1$. In this case, $p \mid p^{n-1}a$, but $p \nmid q^{m-1}b$ (contradiction). If $p<q$, then $a<b$ and $\gcd(q,a)=1$. Thus, $q \mid q^{m-1}b$, but $q \nmid p^{n-1}a$ (contradiction). Hence, $m=1$. Furthermore, $p^n-p^{n-1}=q-1 \iff q=p^n-p^{n-1}+1$. Therefore, $p^n-p^{n-1}+1$ must be a prime number. On the other hand, if $q=p^n-p^{n-1}+1$, where $n>1$, then $q$ is odd prime number and $p^n-p^{n-1}=q-1=q^m-q^{m-1}$, where $m=1$. Using Theorem \ref{lemaisoZpnZqm}, we obtain $Cl_2(\mathbb{Z}_{p^n}) \cong Cl_2(\mathbb{Z}_{q^m})$.
\end{proof}

In the study by \cite{pongthana}, necessary and sufficient conditions were established for $Cl_2(\mathbb{Z}_{p^n r})\cong Cl_2(\mathbb{Z}_{q^m r})$, where $p$ and $q$ are distinct primes and $n,m,r$ are natural numbers such that $\gcd(p,r)=\gcd(q,r)=1$. Here, we discuss a gene\-ralization of that result, which is presented in the following theorem and corollary.
\begin{theorem}\label{teoisoZpnkZqmk}
    Given rings $\mathbb{Z}_{p^n}$ and $\mathbb{Z}_{q^m}$ with $p,q$ are prime numbers and $n,m$ are natural numbers such that $Cl_2(\mathbb{Z}_{p^n}) \cong Cl_2(\mathbb{Z}_{q^m})$, the following holds:
    \begin{align*}
        Cl_2(\mathbb{Z}_{p^n} \times \mathbb{Z}_k) \cong Cl_2(\mathbb{Z}_{q^m} \times \mathbb{Z}_k)
    \end{align*}
    for any natural number $k$.
\end{theorem}
\begin{proof}
    Consider that
    \begin{align*}
        Id(\mathbb{Z}_{p^n}) \setminus \{0\} =  Id(\mathbb{Z}_{q^m}) \setminus \{0\} = \{1\}.
    \end{align*}
    Let
    \begin{align*}
        U(\mathbb{Z}_{p^n})&=\{1,p^n-1,c_1,c_2,\dots,c_t\}\\
        U(\mathbb{Z}_{q^m})&=\{1,q^m-1,d_1,d_2,\dots,d_t\}
    \end{align*} where $t=p^n-p^{n-1}-2$ are even number with $c_ic_{t+1-i}=1$ and $d_id_{t+1-i}=1$ for each $i \in \{1,2,\dots,t\}$.
    We get
    \begin{align*}
        V(Cl_2(\mathbb{Z}_{p^n}))&=\{(1,1),(1,p^n-1),(1,c_i): i \in \{1,2,\dots,t\} \}\\
        V(Cl_2(\mathbb{Z}_{q^m}))&=\{(1,1),(1,q^m-1),(1,d_i): i \in \{1,2,\dots,t\} \}.
    \end{align*}
   Regarding the isomorphism of the graphs $Cl_2(\mathbb{Z}_{p^n})$ and $Cl_2(\mathbb{Z}_{q^m})$, there exists a bijective function $f: V(Cl_2(\mathbb{Z}_{p^n})) \to V(Cl_2(\mathbb{Z}_{q^m}))$ such that 
    $f(1,1)=(1,1)$, $f(1,p^n-1)=f(1,q^m-1)$, and $f(1,c_i)=(1,d_i)$ for every $i \in \{1,2,\dots,t\}$.\\ \hspace{0.4cm}
    Next, let $\mathbb{Z}_k$ be a ring with a natural number $k$. Let
    \begin{align*}
        Id(\mathbb{Z}_k) \setminus \{0\} &= \{e_1=1,e_2,e_3,\dots,e_a\}\\
        U(\mathbb{Z}_k) &= \{u_1=1,u_2,u_3,\dots,u_b\}
    \end{align*} for some $a,b \in \mathbb{Z}^+$. Consider that
    \begin{align*}
        V(Cl_2(\mathbb{Z}_{p^n} \times \mathbb{Z}_k)) = Id(\mathbb{Z}_{p^n} \times \mathbb{Z}_k) \setminus \{(0,0)\} \times U(\mathbb{Z}_{p^n} \times \mathbb{Z}_k)\\
        V(Cl_2(\mathbb{Z}_{q^m} \times \mathbb{Z}_k)) = Id(\mathbb{Z}_{q^m} \times \mathbb{Z}_k) \setminus \{(0,0)\} \times U(\mathbb{Z}_{q^m} \times \mathbb{Z}_k)
    \end{align*}
    with
    \begin{align*}
        &Id(\mathbb{Z}_{p^n} \times \mathbb{Z}_k) \setminus \{(0,0)\} = \{(1,0),(0,e_i),(1,e_i): i \in \{1,2,\dots,a\}\}\\
        &U(\mathbb{Z}_{p^n} \times \mathbb{Z}_k) = \{(1,u_i), (p^n-1, u_i), (c_j,u_i): i \in \{1,2,\dots,b\}, j \in \{1,2,\dots,t\}\}\\
        &Id(\mathbb{Z}_{q^m} \times \mathbb{Z}_k) \setminus \{(0,0)\} = \{(1,0),(0,e_i),(1,e_i): i \in \{1,2,\dots,a\}\}\\
        &U(\mathbb{Z}_{q^m} \times \mathbb{Z}_k) = \{(1,u_i), (q^m-1, u_i), (d_j,u_i): i \in \{1,2,\dots,b\}, j \in \{1,2,\dots,t\}\}.
    \end{align*}
    A bijective function $\varphi: V(Cl_2(\mathbb{Z}_{p^n} \times \mathbb{Z}_k)) \to V(Cl_2(\mathbb{Z}_{q^m} \times \mathbb{Z}_k))$ is constructed with
    \begin{align*}
        \varphi((1,0),(1,u_i))&=((1,0),(1,u_i))\\
        \varphi((1,0),(p^n-1,u_i))&=((1,0),(q^m-1,u_i))\\
        \varphi((1,0),(c_j,u_i))&=((1,0),(d_j,u_i))\\
        \varphi((0,e_l),(1,u_i))&=((0,e_l),(1,u_i))\\
        \varphi((0,e_l),(p^n-1,u_i))&=((0,e_l),(q^m-1,u_i))\\
        \varphi((0,e_l),(c_j,u_i))&=((0,e_l),(d_j,u_i))\\
        \varphi((1,e_l),(1,u_i))&=((1,e_l),(1,u_i))\\
        \varphi((1,e_l),(p^n-1,u_i))&=((1,e_l),(q^m-1,u_i))\\
        \varphi((1,e_l),(c_j,u_i))&=((1,e_l),(d_j,u_i))
    \end{align*} for any $l \in \{1,2,\dots,a\}$, $j \in \{1,2,\dots,t\}$, and $i \in \{1,2,\dots,b\}$.\\
    Let $((x_1,x_2),(y_1,y_2)),((x_3,x_4),(y_3,y_4)) \in V(Cl_2(\mathbb{Z}_{p^n} \times \mathbb{Z}_k))$ be arbitrary such that $((x_1,x_2),(y_1,y_2))((x_3,x_4),(y_3,y_4)) \in E(Cl_2(\mathbb{Z}_{p^n} \times \mathbb{Z}_k))$. It means
    \begin{align*}
        &(x_1,x_2)(x_3,x_4)=(0,0) \text{ or } (y_1,y_2)(y_3,y_4)=(1,1)\\
        \iff &(x_1x_3=0 \text{ , } x_2x_4=0) \text{ or } (y_1y_3=1 \text{ , } y_2y_4=1).
    \end{align*}
    Since 
        \begin{align*}
        \varphi((x_1,x_2),(y_1,y_2))&=((x_1,x_2),(y'_1,y_2)) \text{ and }\\ \varphi((x_3,x_4),(y_3,y_4))&=((x_3,x_4),(y'_3,y_4))
        \end{align*} for any $y'_1,y'_3 \in U(\mathbb{Z}_{q^m})$, the following two cases are considered.
    \begin{enumerate}
        \item Case $x_1x_3=0$ and $x_2x_4=0$. We get $$\varphi((x_1,x_2),(y_1,y_2))\varphi((x_3,x_4),(y_3,y_4)) \in E(Cl_2(\mathbb{Z}_{q^m} \times \mathbb{Z}_k)).$$
        \item Case $y_1y_3=1$ and $y_2y_4=1$. 
        \begin{itemize}
            \item [a. ] If $y_1=1$, then $y_3=1$. Consequently, $y'_1=y'_3=1 \iff y'_1y'_3=1$.
            \item [b. ] If $y_1=p^n-1$, then $y_3=p^n-1$. Consequently, $y'_1=y'_3=q^m-1 \iff y'_1y'_3=1$.
            \item [c. ] If $y_1=c_i$, then $y_3=c_{t+1-i}$ where $i \in \{1,2,\dots,t\}$. Hence $y'_1=d_i$ and $y'_3=d_{t+1-i}$, so $y'_1y'_3=1$.
        \end{itemize}
        Thus, $\varphi((x_1,x_2),(y_1,y_2))\varphi((x_3,x_4),(y_3,y_4)) \in E(Cl_2(\mathbb{Z}_{q^m} \times \mathbb{Z}_k)).$
    \end{enumerate}
    Let $((x_1,x_2),(y_1,y_2)),((x_3,x_4),(y_3,y_4)) \in V(Cl_2(\mathbb{Z}_{p^n} \times \mathbb{Z}_k))$ be arbitrary such that $\varphi((x_1,x_2),(y_1,y_2))\varphi((x_3,x_4),(y_3,y_4)) \in E(Cl_2(\mathbb{Z}_{q^m} \times \mathbb{Z}_k))$. Since 
        \begin{align*}
        \varphi((x_1,x_2),(y_1,y_2))&=((x_1,x_2),(y'_1,y_2)) \text{ and }\\ \varphi((x_3,x_4),(y_3,y_4))&=((x_3,x_4),(y'_3,y_4))
        \end{align*} for some $y'_1,y'_3 \in U(\mathbb{Z}_{q^m})$, we get
    \begin{align*}
        &(x_1,x_2)(x_3,x_4)=(0,0) \text{ or } (y'_1,y_2)(y'_3,y_4)=(1,1)\\
        \iff &(x_1x_3=0 \text{ , } x_2x_4=0) \text{ or } (y'_1y'_3=1 \text{ , } y_2y_4=1).
    \end{align*}
        Next, the following two cases are considered.
    \begin{enumerate}
        \item Case $x_1x_3=0$ and $x_2x_4=0$. Thus $$((x_1,x_2),(y_1,y_2))((x_3,x_4),(y_3,y_4)) \in E(Cl_2(\mathbb{Z}_{p^n} \times \mathbb{Z}_k)).$$
        \item Case $y'_1y'_3=1$ and $y_2y_4=1$. 
        \begin{itemize}
            \item [a. ] If $y'_1=1$, then $y'_3=1$. Consequently, $y_1=y_3=1 \iff y_1y_3=1$.
            \item [b. ] If $y'_1=q^m-1$, then $y'_3=q^m-1$. Consequently, $y_1=y_3=p^n-1 \iff y_1y_3=1$.
            \item [c. ] If $y'_1=d_i$, then $y'_3=d_{t+1-i}$ where $i \in \{1,2,\dots,t\}$. Hence $y_1=c_i$ and $y_3=c_{t+1-i}$, so $y_1y_3=1$.
        \end{itemize}
    \end{enumerate}
\end{proof}

\begin{corollary}
    Given the rings $\mathbb{Z}_{p^n} \times \mathbb{Z}_k$ and $\mathbb{Z}_{q^m} \times \mathbb{Z}_k$, where $p,q$ are distinct odd prime numbers and $n,m,k$ are positive integers, the following holds:
    \begin{align*}
        Cl_2(\mathbb{Z}_{p^n} \times \mathbb{Z}_k) \cong Cl_2(\mathbb{Z}_{q^m} \times \mathbb{Z}_k) \iff p^n-p^{n-1} = q^m - q^{m-1}.
    \end{align*}
    Furthermore, for $n>1$, we have
    \begin{align*}
    	Cl_2(\mathbb{Z}_{p^n} \times \mathbb{Z}_k) \cong Cl_2(\mathbb{Z}_{q^m} \times \mathbb{Z}_k) \iff &q=p^n-p^{n-1}+1 \text{ is a prime number  } \\ &\text{and } m=1.
    \end{align*}
\end{corollary}
\begin{proof}
    ($\Rightarrow$) Suppose $Cl_2(\mathbb{Z}_{p^n} \times \mathbb{Z}_k) \cong Cl_2(\mathbb{Z}_{q^m} \times \mathbb{Z}_k)$. Using Corollary \ref{equalisoIdU}, we obtain
    \begin{align*}
        &| Id(\mathbb{Z}_{p^n} \times \mathbb{Z}_k)|= |Id(\mathbb{Z}_{q^m} \times \mathbb{Z}_k)| \text{ and } |U(\mathbb{Z}_{p^n} \times \mathbb{Z}_k)|=|U(\mathbb{Z}_{q^m} \times \mathbb{Z}_k)|\\
        \Longrightarrow &|U(\mathbb{Z}_{p^n}) \times U(\mathbb{Z}_k)|=|U(\mathbb{Z}_{q^m}) \times U(\mathbb{Z}_k)|\\
        \iff &|U(\mathbb{Z}_{p^n})|  |U(\mathbb{Z}_k)|=|U(\mathbb{Z}_{q^m})|  |U(\mathbb{Z}_k)|.
    \end{align*}
    Since $|U(\mathbb{Z}_k)| \geq 1$, we obtain $|U(\mathbb{Z}_{p^n})|=|U(\mathbb{Z}_{q^m})|$. Thus,
        $$p^n\left(1-\frac{1}{p}\right)= q^m\left(1-\frac{1}{q}\right) \iff p^n-p^{n-1}=q^m-q^{m-1}.$$
        Furthermore, since $p,q$ are odd prime numbers, at least $n>1$ or $m>1$. Without loss of generality, let $n>1$. Using Theorem \ref{teocorZpnpm}, $q=p^n-p^{n-1}+1$ and $m=1$.\\
    ($\Leftarrow$) Suppose $p$ and $q$ are odd prime numbers with $p^n-p^{n-1}=q^m-q^{m-1}$. Using Lemma \ref{lemaisoZpnZqm}, we obtain $Cl_2(\mathbb{Z}_{p^n}) \cong Cl_2(\mathbb{Z}_{q^m})$. Based on Theorem \ref{teoisoZpnkZqmk}, we conclude that $Cl_2(\mathbb{Z}_{p^n} \times \mathbb{Z}_k) \cong Cl_2(\mathbb{Z}_{q^m} \times \mathbb{Z}_k)$. In a more specific case, where $q=p^n-p^{n-1}+1$ and $m=1$, we also obtain $p^n-p^{n-1}=q^m-q^{m-1}$. Using Lemma \ref{lemaisoZpnZqm}, $Cl_2(\mathbb{Z}_{p^n}) \cong Cl_2(\mathbb{Z}_{q^m})$. Based on Theorem \ref{teoisoZpnkZqmk}, $Cl_2(\mathbb{Z}_{p^n} \times \mathbb{Z}_k) \cong Cl_2(\mathbb{Z}_{q^m} \times \mathbb{Z}_k)$.
\end{proof}

From Theorem \ref{teoisoZpnkZqmk}, when considered in the general case for an arbitrary ring, we obtain the following conjecture.
\begin{conjecture}
    Given the rings $R_1, R_2, P_1,$ and $P_2$ such that $Cl_2(R_1) \cong Cl_2(R_2)$ dan $P_1 \cong P_2$. The following holds: 
    $$Cl_2(R_1 \times P_1) \cong Cl_2(R_2 \times P_2).$$
\end{conjecture}

\section{The Structural Insight into $M_2(\mathbb{Z}_p)$}

We consider the clean graph $Cl_2$ over the ring $M_{2}(\mathbb{Z}_p)$, where $p$ is a prime number. Since $\mathbb{Z}_p$ is a field for any prime $p$, applying Theorem 6 in \cite{djuang} and Proposition 3.7 in \cite{patil} , we obtain the following corollary:

\begin{corollary}
Let $p$ be a prime number. Then,

$$
Cl_2(M_{2}(\mathbb{Z}_p)) \cong 
\begin{cases}
Shu^4_6(3K_2), & \text{if } p = 2, \\
Shu^{p^2 + p + 2}_{p^4 - p^3 - p^2 + p}\left(\frac{p(p+1)}{2}K_2 \right), & \text{if } p > 2.
\end{cases}
$$

\end{corollary}

\begin{proof}
    Theorem 6 in \cite{djuang} and Proposition 3.7 in \cite{patil}, we have $$
Cl_2(M_{2}(\mathbb{Z}_p)) \cong Shu^{t}_{n}\left(\frac{p(p+1)}{2}K_2 \right),$$
where $t = |U'(M_{2}(\mathbb{Z}_p))|$ and $n = |U(M_{2}(\mathbb{Z}_p))|$. We first examine the set $U(M_{2}(\mathbb{Z}_p))$. For any matrix $\begin{bmatrix} a & b\\  c & d
\end{bmatrix} \in M_{2}(\mathbb{Z}_p)$, it holds that:
\begin{align*}
    \begin{bmatrix}
    a & b\\
    c & d
    \end{bmatrix} \in U(M_{2}(\mathbb{Z}_p))
&\iff \det\begin{bmatrix}
a & b\\
c & d
\end{bmatrix} \in U(\mathbb{Z}_p) = \mathbb{Z}_p \setminus \{0\} \\
&\iff \det\begin{bmatrix}
a & b\\
c & d
\end{bmatrix} \ne 0 \\
&\iff \text{rank}\begin{bmatrix}
a & b\\
c & d
\end{bmatrix} = 2 \\
&\iff
\begin{bmatrix}
c & d
\end{bmatrix} \ne k
\begin{bmatrix}
a & b
\end{bmatrix}, \ \forall k \in \mathbb{Z}_p, \\& \quad \qquad \text{ and }
\begin{bmatrix}
a & b
\end{bmatrix} \ne
\begin{bmatrix}
0 & 0
\end{bmatrix}.
\end{align*}

There are $p^2 - 1$ nonzero vectors for $\begin{bmatrix} a & b \end{bmatrix}$, and for each such choice, there are $p^2 - p$ linearly independent choices for $\begin{bmatrix} c & d \end{bmatrix}$. Therefore,

$$
n = |U(M_{2}(\mathbb{Z}_p))| = (p^2 - 1)(p^2 - p) = p^4 - p^3 - p^2 + p.
$$

Next, we consider the set $U'(M_{2}(\mathbb{Z}_p))$. Let

$$
\begin{bmatrix}
a & b\\ 
c & d
\end{bmatrix} \in U'(M_{2}(\mathbb{Z}_p)).
$$

Then it satisfies:
\begin{align*}
\begin{bmatrix}
a & b\\
c & d
\end{bmatrix}^2 =
\begin{bmatrix}
1 & 0\\
0 & 1
\end{bmatrix}
&\iff
\begin{bmatrix}
a^2 + bc & ab + bd\\
ac + cd & bc + d^2
\end{bmatrix}=
\begin{bmatrix}
1 & 0\\
0 & 1
\end{bmatrix} \\
&\iff
\begin{cases}
a^2 + bc = 1 \\
ab + bd = 0 \\
ac + cd = 0 \\
bc + d^2 = 1
\end{cases}
\iff
\begin{cases}
a^2 + bc = 1 \\
b(a + d) = 0 \\
c(a + d) = 0 \\
bc + d^2 = 1
\end{cases}
\end{align*}

Since $\mathbb{Z}_p$ is a field, the conditions $b(a+d) = 0$ and $c(a+d) = 0$ yield the following four cases:
\begin{enumerate}
    \item Case $b = 0$ and $c = 0$.
   Then $a^2 = 1$ and $d^2 = 1$, implying $a, d \in U'(\mathbb{Z}_p) = \{1, p-1\}$. Thus,
   $$ \begin{bmatrix}
    a & 0\\ 
    0 & d
    \end{bmatrix}, \text{ where } a, d \in \{1, p-1\}.$$
    
    \item Case $b = 0$ and $a = -d$.
   Then again $a^2 = d^2 = 1$, so $(a, d) = (1, p-1)$ or $(p-1, 1)$. Excluding the case $c = 0$ already considered above, we have:
    $$\begin{bmatrix}
        1 & 0\\ 
        c & p-1
    \end{bmatrix}, 
    \begin{bmatrix}
        p-1 & 0\\ 
        c & 1
    \end{bmatrix}, \text{ where } c \in \mathbb{Z}_p \setminus \{0\}.$$

    \item Case $c = 0$ and $a = -d$.
   Similarly, excluding the case $b = 0$, we obtain:
   $$\begin{bmatrix}
        1 & b\\ 
        0 & p-1
    \end{bmatrix}, 
    \begin{bmatrix}
        p-1 & b\\ 
        0 & 1
    \end{bmatrix}, \text{ where } b \in \mathbb{Z}_p \setminus \{0\}.$$

    \item Case $a = -d$, with $b \ne 0$, $c \ne 0$.
   Then $a^2 = 1 - bc$, so $c = \frac{1 - a^2}{b}$. Since $c \ne 0$, $b \ne 0$, we must have $1 - a^2 \ne 0 \Rightarrow a \notin \{1, p-1\}$. Thus:
    $$\begin{bmatrix}
        a & b\\ 
        \frac{1 - a^2}{b} & -a
    \end{bmatrix}, \text{ where } a \in \mathbb{Z}_p \setminus \{1, p-1\}, \ b \in \mathbb{Z}_p \setminus \{0\}.$$
\end{enumerate}
Combining all cases, the cardinality of $U'(M_{2}(\mathbb{Z}_p))$ is:
\begin{align*}
    t = |U'(M_{2}(\mathbb{Z}_p))| &= 
    \begin{cases}
        4, & \text{if } p = 2, \\
        4 + 2(p - 1) + 2(p - 1) + (p - 2)(p - 1), & \text{if } p > 2
    \end{cases}\\
    &= \begin{cases}
        4, & \text{if } p = 2, \\
        p^2 + p + 2, & \text{if } p > 2.
    \end{cases}
\end{align*}

Therefore, we conclude:

$$
Cl_2(M_{2}(\mathbb{Z}_p)) \cong 
\begin{cases}
Shu^4_6(3K_2), & \text{if } p = 2, \\
Shu^{p^2 + p + 2}_{p^4 - p^3 - p^2 + p}\left(\frac{p(p+1)}{2}K_2 \right), & \text{if } p > 2.
\end{cases}
$$

\end{proof}

\end{document}